\documentclass[a4paper,10pt,reqno]{amsart}

\usepackage{todonotes}
\usepackage{enumitem}
\usepackage{amsmath}
\usepackage{amsthm}
\usepackage{amssymb}
\usepackage{pdfpages}
\usepackage{graphicx}
\usepackage{esint}

\title{Polynomial upper bound on interior Steklov nodal sets}

\newtheorem{claim}{Claim}[section]
\newtheorem{theorem}{Theorem}[section]
\newtheorem{lemma}{Lemma}[section]
\newtheorem{definition}{Definition}[section]

\newtheorem{remark}{Remark}[section]
\newtheorem{proposition}{Proposition}[section]

\DeclareMathOperator{\diam}{diam}
\DeclareMathOperator{\width}{width}

\author{Bogdan Georgiev}
\address{Max Planck Institute for Mathematics, Vivatsgasse 7, 53111 Bonn, Germany} 
\email{bogeor@mpim-bonn.mpg.de}

\author{Guillaume Roy-Fortin}
\address{Department of Mathematics, Northwestern  University,
Evanston, IL 60208-2370, USA} 
\email{gui@math.northwestern.edu}

\begin{document}
	\maketitle
	\begin{abstract}
		We study solutions of uniformly elliptic PDE with Lipschitz leading coefficients and bounded lower order coefficients. We extend previous results of A. Logunov (\cite{L}) concerning nodal sets of harmonic functions and, in particular, prove polynomial upper bounds on interior nodal sets of Steklov eigenfunctions in terms of the corresponding eigenvalue $ \lambda $.\\
	\end{abstract}
	
	\section{Introduction}
	
	This paper considers non trivial solutions $u$ to the following general second order elliptic equation 
	\begin{equation}\label{eq_main_elliptic}
		Lu := \sum_{i,j=1}^N \frac{\partial}{\partial x_i}\left(a^{ij}(x) \frac{\partial u}{\partial x_j}\right) + \sum_{i=i}^N b^i(x) \frac{\partial u}{\partial x_j} + c(x) u = 0,
	\end{equation}
	in some smooth bounded domain $\Omega \subset \mathbb{R}^n$. We make the following assumptions on the coefficients of $L$:
	\begin{enumerate}
		\item $ L $ is uniformly elliptic, that is for a fixed $ \eta > 0 $ we have
		\begin{equation} \label{eq:Uniformly-Elliptic}
			a^{ij}(x) \xi_i \xi_j \geq \eta |\xi|^2, \quad \forall \xi \in \mathbb{R}^n, x \in \Omega.
		\end{equation}
		
		\item The coefficients of $ L $ are bounded
		\begin{equation} \label{eq:Uniformly-Bounded}
			\sum_{i, j} |a^{ij}(x)| + \sum_i |b^i(x)| +|c(x)| \leq \Lambda , \quad x \in \Omega.
		\end{equation}
		
		\item The leading coefficients are Lipschitz
		\begin{equation} \label{eq:Lipschitz-Coef}
			\sum_{ij} |a^{ij}(x) - a^{ij}(y)| \leq \Gamma |x-y|.
		\end{equation}
	\end{enumerate}
	
	We focus our interest on the relation between the zero set and the local growth properties of a solution $u$.
	
	\subsection{Doubling indices and nodal set}
	
	Given a fixed ball $B$ such that $2B \subset \Omega$, the \textit{doubling index} $N(B)$ is a measure of the local growth of $u$ on $B$ defined by 
	\begin{equation}\label{def:Doubling-Index}
	N(B) := \log \frac{\sup_{2B} |u|}{\sup_{B} |u|}
	\end{equation}
	Here and for the rest of the paper $rB$ is the ball concentric to $B$ and scaled by a factor $r > 0$. As the following simple example shows, the doubling index can be seen as a local generalization of the degree of a polynomial for continuous functions. Letting $u=x^n$ and $B = [-r,r]$, we have 
	\begin{equation*}
	N(B) = \log \frac{\sup_{[-2r, 2r]} |x|^n}{\sup_{[-r,r]} |x|^n} = \log \frac{(2r)^n}{r^n} = n (\log 2).
	\end{equation*}
	Thus, the doubling index indeed recovers the degree up to a constant. We will often write $N(x,r)$ for the doubling index of $u$ on the ball $B(x,r)$.\\

	The \textit{nodal set} of $u$ is simply its zero set $$Z_u = \{ u^{-1}(0) \}.$$ Sparked by the famous conjecture of Yau \cite{Ya1, Ya2} on nodal sets of Laplace eigenfunctions, it is a celebrated problem to try to estimate the Hausdorff measure $\mathcal{H}^{n-1}(Z_u)$ of the nodal set of solutions to various partial differential equations. By the work \cite{HS} of Hardt-Simon, it is known that $\mathcal{H}^{n-1}(Z_u)$ is finite. \\

	The seminal papers by Donnelly-Fefferman \cite{DF} (see also the more recent work \cite{RF1, RF2} of the second author ) highlight how the doubling index can be used to obtain bounds on the size of the nodal set. Our main result is along these lines and extends the work of Logunov \cite{L} for harmonic functions to solutions $u$ of Equation (\ref{eq_main_elliptic}). More precisely, we show that the size of the nodal set of such solutions is controlled by the doubling index in the following way:
	
	\begin{theorem} \label{thm:Main-Thm}
There exist positive numbers $ r_0 = r_0(M, g), c = c(M, g) $ and $ \alpha = \alpha(n) $ such that for any solution $ u $ of equation $ (\ref{eq_main_elliptic}) $ in a domain $ \Omega $ satisfying the conditions $ (\ref{eq:Uniformly-Elliptic}), (\ref{eq:Uniformly-Bounded}), (\ref{eq:Lipschitz-Coef})$, we have
		\begin{equation}
			\mathcal{H}^{n-1}(Z_u \cap Q) \leq c \diam^{n-1}(Q) N^\alpha (Q),
		\end{equation}
		where $ Q \subset B(p, r_0) $ is an arbitrary cube in $\Omega$.
	\end{theorem}
	
	Here, $N(Q)$ is the uniform doubling index of $u$ on a cube $Q$ as defined by 
	\begin{equation}
	N(Q) := \sup_{x \in Q, r \in (0, \text{diam}(Q))} N(x,r).
	\end{equation}

	The proof adapts the machinery developed by Logunov to solutions of more general elliptic equations. In Section \ref{sec:Tool-Box}, we build a toolbox consisting mostly of elliptic estimates and almost monotonicity of a generalized frequency function - see Equation (\ref{eq:frequency-fct}) - that we then use in Section \ref{sec:Additivity-Frequency} to prove our generalized versions of the crucial simplex and hyperplane lemmata. Those two lemmata work together to investigate the additivity properties of the frequency. The underlying principal idea can be roughly summarized as follows: if the frequency of $u$ on a big cube $Q$ is high, then it cannot be high in too many disjoint sub-cubes $q_i \subset Q$.
		
	\subsection{Application: interior nodal sets of Steklov eigenfunctions}
	
		Let $M$ be a smooth, connected and compact manifold of dimension $n \geq 2$ with non-empty smooth boundary $\partial M$ and denote by $\Delta = \Delta_g$ the Laplace-Beltrami operator on $M$. The Steklov eigenfunctions on $M$ are solutions to 
		\begin{equation}\label{Eq:main-Steklov}
			\begin{cases}
				\Delta \phi = 0 &  \text{ in } M,\\
				\partial_\nu \phi = \lambda \phi & \text{ on } \partial M.
			\end{cases}
		\end{equation} 
		
		In this setting, the spectrum is discrete and is composed of the eigenvalues $$ 0 = \lambda_0 < \lambda_1 \leq \lambda_2 \leq ... \nearrow \infty.$$
		Given a Steklov eigenfunction $u = u_\lambda$, we distinguish the codimension $1$ \textit{interior nodal set} 
		\begin{equation}\label{eq:interior-Nodal-set}
			Z_\lambda = \left\{ x \in M : \phi(x) = 0 \right\}
		\end{equation}
		and the codimension $2$ \textit{boundary nodal set}  
		\begin{equation}\label{eq:bdy-Nodal-set}
			N_\lambda = \left\{ p \in \partial M : \phi(p) = 0 \right\}.
		\end{equation}
		As mentioned earlier, we are interested in measuring the size of these nodal sets. It is expected (see \cite{GP}) that their size is controlled by the Steklov eigenvalue. More precisely, it is conjectured that
		\begin{equation}\label{eq_yau_interior} 
			c_1 \lambda \leq \mathcal{H}^{n-1}(Z_\lambda) \leq c_2 \lambda
		\end{equation} and 
		\begin{equation}\label{eq_yau_bdy}
			c_3 \lambda \leq \mathcal{H}^{n-2}(N_\lambda) \leq c_4 \lambda, 
		\end{equation} 
		where $\mathcal{H}^n$ is the $n$-dimensional Hausdorff measure. In the above, the $c_i$ are positive constants that may only depend on the geometry of the manifold $M$. These conjectures are similar to the famous Yau conjecture for nodal sets of 			eigenfunctions of the Laplace operator. We now briefly present the current best results present in the literature, starting with the interior nodal set:

	\begin{table}[h]
	\centering
	\caption{Current best bounds for $\mathcal{H}^{n-1}(Z_\lambda)$}
	\label{my-label}
		\begin{tabular}{c|c|c}
			Regularity and dimension & Current Best Lower Bound         & Current Best Upper Bound      \\ \hline
			$C^\omega$, $n=2$        & $c \lambda$ \cite{PST} $\checkmark$    & $c \lambda$ \cite{PST} $\checkmark$ \\
			$C^\omega$, $n \geq 3$        &  $c \lambda^{\frac{2-n}{2}}$ \cite{SWZ} &  \\ 
			$C^\infty$, $n = 2$   &  c \cite{SWZ} & $c \lambda^\frac{3}{2}$ \cite{Zhu1}  \\       
			$C^\infty$, $n \geq 3$   & $c \lambda^{\frac{2-n}{2}}$ \cite{SWZ} &                              
	\end{tabular}
	\end{table}

	\bigskip
	In the case of the boundary nodal set, we have

	\begin{table}[h]
	\centering
	\caption{Current best bounds for $\mathcal{H}^{n-2}(N_\lambda)$}
	\label{my-label}
		\begin{tabular}{c|c|c}
			Regularity and dimension & Current Best Lower Bound         & Current Best Upper Bound      \\ \hline
			$C^\omega$, $n \geq 2$        &    & $c \lambda$ \cite{Ze} $\checkmark$ \\
			$C^\infty$, $n = 2$   &  $c \lambda$ \cite{WZ} &   \\       
			$C^\infty$, $n \geq 3$   & $c \lambda^\frac{4-n}{2}$ \cite{WZ} &                              
		\end{tabular}
	\end{table}

	We use Theorem \ref{thm:Main-Thm} to provide a polynomial upper bound for interior nodal sets in the smooth case in any dimension $n \geq 2$.

	\begin{theorem} \label{thm:Steklov-Bound}
		 Let $M$ be a smooth, connected and compact manifold of dimension $n \geq 2$ with non-empty smooth boundary $\partial M$. Let $\phi_\lambda $ be a Steklov eigenfunction on $M$ corresponding to the eigenvalue $\lambda$. Then
		\begin{equation}
			\mathcal{H}^{n-1}(Z_\lambda) \leq c \lambda^\alpha,
		\end{equation}
		where $ c = c(M, g) $ and $ \alpha = \alpha(n) $.
	\end{theorem}

	The proof is based on a gluing procedure that transforms $M$ into a compact manifold without boundary. Doing so and working locally then allows to transfer the study of the nodal set of $\phi$ to that of a solution $u$ to the elliptic Equation \ref{eq_main_elliptic}. The details are presented in Section \ref{sec:Steklov}.
	
	\begin{remark}
	Let us finally notice that the methods of this paper could be applied directly to get a similar polynomial upper bound $\lambda^\alpha$ for the nodal sets of Laplace eigenfunctions on a smooth compact $n$-manifold $M$, $n \geq 2$, with smooth boundary.
	\end{remark}
	
	\begin{remark}
	Constants are labeled $c_1, c_2, ...$ and we make the decision to keep track of them throughout the article. Although this makes the notation heavy, this facilitates tracking down the explicit values of these constants. This in turn should make easier the task of getting some upper bound on the exponent $\alpha$ of  Theorem \ref{thm:Main-Thm}, a question that the authors wish to investigate in the future.
	\end{remark}
	
	\subsection{Acknowledgements}
	 The authors are grateful to Werner Ballmann, Alexander Logunov, Eugenia Malinnikova, Iosif Polterovich and Steve Zelditch for comments, valuable discussions and feedback on this manuscript.

	\section{Tool box} \label{sec:Tool-Box}
	
	\subsection{Elliptic estimates}
	
	We recall Theorem 8.24 of \cite{GT} for operators of the type $ L $ as above. For any weak solution $ u \in W^{1,2}$ of $Lu = 0 $ and $\epsilon > 0$, we have the following elliptic estimate
	\begin{equation}\label{eq:Elliptic-Regularity}
		\sup_{B(x, \rho)} |u|^2 \leq c_1 \fint_{B(x, (1+\epsilon)\rho)} u^2,
	\end{equation}
	where $c_1 = c_1(n, L, \epsilon)$. On the other hand, for every continuous function
	\begin{equation}\label{eq:sup-estimate}
		\fint_{B(x, \rho)} u^2 \leq \sup_{B(x, \rho)} |u|^2.
	\end{equation}

	\subsection{Properties of the frequency function}
	
	Let $ w \in W^{1, 2}_{loc} (B_1) $. We define
	\begin{equation}
		H(r) := \int_{\partial B_r} w^2 d \sigma, \quad D(r) := \int_{B_r} |\nabla w|^2 dx,
	\end{equation}
	
	\begin{equation}
		I(r) := \int_{B_r} (|\nabla w|^2 + w (\textbf{b} \cdot \nabla w) + cw^2 ) dx.
	\end{equation}
	
	The generalized frequency $ \beta(r) $ is defined as
	\begin{equation}\label{eq:frequency-fct}
		\beta(r) := \frac{r I(r)}{H(r)}.
	\end{equation}
	
	The frequency $\beta$ enjoys an almost monotonicity property:
	
	\begin{theorem}[cf. Theorem 3.2.1, \cite{HL}, Theorem 3, Proposition 17 \cite{BL}] \label{thm:Monotonicity}
		There exist constants $ c_2, c_3 > 0 $ such that for any $ u \in W^{1, 2}_{loc}(B_1) $ and  $ Lu = 0 $, we have
		\begin{equation}
			\beta(r) \leq c_2 + c_3 \beta(r_0), \quad r \in (0, r_0).
		\end{equation}
		Moreover, if $ r_0 $ is sufficiently small then $ c_3 $ can be taken to be $ 1 + \epsilon, \epsilon > 0 $.
	\end{theorem}
	
	The second statement of the Theorem can be verified by inspection of the end of the proof in \cite{HL}, noticing that $c_3$ can be chosen to be $e^{c(r-r_0)}$. We also have the following derivation formula (cf. Corollary 3.2.8 in \cite{HL})
	\begin{equation}
		\frac{d}{dr}\left( \log \frac{H(r)}{r^{n-1}} \right) = O(1) + 2 \frac{\beta(r)}{r} \geq - c_4,
	\end{equation}
	where $c_4 = c_4(n) > 0$. As a consequence, we get
	
	
	\begin{lemma}\label{lemma_monotonicity_Hr}
	The function $\displaystyle \frac{e^{c_4r} H(r)}{r^{n-1}}$ is increasing for $r \in (0, r_0)$. 
	\end{lemma}
	
	Now let $0 < R_1 < R_2 <Êr_0$. An integration yields	
	\begin{equation}
		H(R_2) = H(R_1) \left( \frac{R_2}{R_1}\right)^{n-1} \exp \left( O(1)(R_2 - R_1) + 2 \int_{R_1}^{R_2} \frac{\beta(r)}{r} dr \right).
	\end{equation}
	
	Using the almost monotonicity of the frequency, we estimate the integral on the right hand side by 
	\begin{equation}
	\log \left(\frac{R_2}{R_1}\right) c_3^{-1}(\beta(R_1) - c_2) \leq \int_{R_1}^{R_2} \frac{\beta(r)}{r} dr \leq \log \left(\frac{R_2}{R_1}\right)(c_2 + c_3 \beta(R_2)),
	\end{equation}
	
	which, after absorbing the dimensional constants, yields
	\begin{equation}\label{eq:Frequency-Scaling}
		 \left( \frac{R_2}{R_1} \right)^{ c_3^{-1}( \beta(R_1) - c_2)} \leq \frac{H(R_2)}{H(R_1)} \leq c_5 \left( \frac{R_2}{R_1} \right)^{2(c_3 \beta(R_2) + c_2)}.
	\end{equation}

	\subsection{Doubling numbers and scaling}
	
	
	 The main technical tool we need is
	
	\begin{lemma} \label{lem:Doubling-Scaling}
		Let $ \epsilon \in (0, 1) $. There exist positive constants $ c_6 = c_6(\epsilon) $ and $ r_0 = r_0(\epsilon) $, such that for any  $ u \in W^{1,2}(B)$ with $Lu = 0 $, we have
		\begin{equation}
			t^{N(x, \rho) (1-\epsilon) - c_6} \leq \frac{\sup_{B(x, t\rho)} |u|}{\sup_{B(x, \rho)} |u|} \leq t^{N(x, \rho) (1+\epsilon)+c_6},
		\end{equation}
		for any $ \rho > 0, t > 2 $. Furthermore, there is a threshold $ N_0 = N_0 (\epsilon) $, such that if $ N(x, \rho) > N_0 $, then the constant $c_6 $ can be dropped in the above estimate and one has
		\begin{equation}
			t^{N(x, \rho)(1 - \epsilon)} \leq \frac{\sup_{B(x, t\rho)} |u|}{\sup_{B(x, \rho)} |u|} \leq t^{N(x, t\rho)(1+\epsilon) }.
		\end{equation}
	\end{lemma}
	
	\begin{proof}
		The argument goes along the lines of the Appendix in \cite{L} with appropriate modifications. For completeness we provide the technical details. We prove the following claim.
		
		\begin{claim}
			Suppose $\epsilon > 0$ and $ r_0 > 0$ are sufficiently small. Then
			\begin{equation}
				\beta(p, r(1+\epsilon))(1-100\epsilon) - c_7  \leq N(p,r) \leq \beta(p, 2r(1+\epsilon))(1+100\epsilon) + c_7.
			\end{equation}
		\end{claim}
		
		Using the elliptic estimate (\ref{eq:Elliptic-Regularity}) and Lemma \ref{lemma_monotonicity_Hr}, there exists $\epsilon > 0$ such that
		\begin{equation}\label{eq:Elliptic-RegularityImp}
			\sup_{B(p, r)} |u|^2  \leq c_1 H((1+\epsilon)r) / r^{n-1}.
		\end{equation}
		
		Using Lemma \ref{lemma_monotonicity_Hr}, there holds $H((1- \epsilon) 2r) \leq e^{2 c_4 r}H(2r)$  so that
		\begin{equation} \label{eq:sup-estimateImp}
			\sup_{B(p, 2r)} |u|^2 \geq  \frac{1}{\omega_n (2r)^n} \int_{B(p, 2r)} u^2 \geq \frac{1}{\omega_n (2r)^n} \int_{2r(1-\epsilon)}^{2r} H(\rho) d \rho \geq c_2  \frac{H(2r(1-\epsilon))}{r^{n-1}},
		\end{equation}
		where $\displaystyle c_2(\epsilon, n) = \frac{\epsilon}{\omega_n 2^{n-1}e^{c_4r_0}}$. Using the latter, we estimate the doubling indices as follows
		\begin{equation}
			N(p, r) =  \log \frac{\sup_{B(p, 2r)}|u|}{\sup_{B(p, r)} |u|} \geq \log \frac{ H(2r(1-\epsilon)) }{H(r(1+\epsilon))} + c_8, 
		\end{equation}
		where $\displaystyle c_8 = \log{\frac{c_2}{c_1}}$. The last quotient is controlled via the generalized frequency as given in (\ref{eq:Frequency-Scaling}). Further, assume that $ r_0 $ is sufficiently small, so that $ c_3 = 1 + \epsilon $. Then, we have
		\begin{align}
			 \log \frac{ H(2r(1-\epsilon)) }{H(r(1+\epsilon))} &\geq \log  \left[ \left( \frac{2(1-\epsilon)}{1+\epsilon} \right)^{\frac{\beta(r(1+\epsilon)) - c_2}{1+\epsilon}} \right] \\ &\geq \frac{ \beta(r(1+\epsilon)) - c_2}{1+\epsilon} \log \left( \frac{2(1-\epsilon)}{1+\epsilon}\right)  \\
			& \geq \frac{ \beta(r(1+\epsilon))}{1+\epsilon} \log \left[ \frac{2(1-\epsilon)}{1+\epsilon}\right] - c_9.
		\end{align}
		
		Now, we recall that for small $ r $, the frequency function is "almost non-negative" in the sense that (cf. Corollary 10, \cite{BL})
		\begin{equation}
			\frac{\beta(r)}{r} \geq -c_{10},
		\end{equation}
		where $c_{10} > 0$.
		Thus, for a sufficiently small $ \epsilon > 0 $ we get
		\begin{equation}
			\frac{ \beta(r(1+\epsilon))}{1+\epsilon} \log  \left[ \frac{2(1-\epsilon)}{(1+\epsilon)}\right] - c_9 \geq \beta(r(1+\epsilon))(1-20 \epsilon) - c_{9}.
		\end{equation}
		
		Hence, we obtain 
		\begin{equation}
			N(p, r) \geq \beta(p, r(1+\epsilon)) (1 - 100\epsilon) - c_7.
		\end{equation}
		
		Similarly, one sees
		\begin{equation}
			N(p,r) \leq \beta(p, 2r(1+\epsilon))(1+100\epsilon) + c_7,
		\end{equation}
		
		provided that $ \epsilon $ and $ r_0 $ are sufficiently small. This finishes the proof of the claim.

		We now proceed showing the lower bound in Lemma \ref{lem:Doubling-Scaling}. First, we can assume that $ t $ is bounded away from $ 2 $. Indeed, if $ t \leq 2^{1+\epsilon} $, then as $ t>2 $ we have $ t^{N(x, \rho)(1-\epsilon)} \leq 2^{N(x, \rho)} $. Hence
		\begin{equation}
			\sup_{B(x, t\rho)} |u| \geq \sup_{B(x, 2\rho)} |u| \geq c_{11} 2^{N(x, \rho)} \sup_{B(x, \rho)} |u| \geq t^{N(x, \rho)(1 -\epsilon)} \sup_{B(x, \rho)} |u|,
		\end{equation}
		which gives the lower bound and the additional statement as well. 
		
		So, we assume that $ t > 2^{1+\epsilon} $. Let us also set $ \tilde{\epsilon} := \epsilon/1000 $, so that $ (1-\tilde{\epsilon})t > 2(1 + \tilde{\epsilon}) $. Using the estimates (\ref{eq:Elliptic-RegularityImp}, \ref{eq:sup-estimateImp}), the frequency scaling (\ref{eq:Frequency-Scaling}) and the last claim, we have
		\begin{align}
			\frac{\sup_{B(x, t\rho)} |u|^2}{\sup_{B(x, \rho)} |u|^2} &\geq \frac{ c_2 (t\rho)^{1-n} H((1-\epsilon)t\rho)}{e^{-2N(x,\rho)} \sup_{B(x, 2\rho)} u^2 }  \\ & \geq c_8 \frac{\left( \frac{(1-\tilde{\epsilon})t}{2(1+\tilde{\epsilon})} \right)^{(2N(x, \rho)/(1+100\tilde{\epsilon})(1+\tilde{\epsilon})) - c_9} H(2\rho(1 + \tilde{\epsilon}))}{c_{10}  e^{-2N(x,\rho)} H(2\rho(1+\tilde{\epsilon}))}  \\ & \geq c_{11} e^{2N(x,\rho)} \left( \frac{(1-\tilde{\epsilon})t}{2(1+\tilde{\epsilon})} \right)^{(2N(x, \rho)/(1+100\tilde{\epsilon})(1+\tilde{\epsilon})) - c_9} \\ & \geq c_{12} \left( \frac{(1-\tilde{\epsilon})t}{(1+\tilde{\epsilon})} \right)^{(2N(x, \rho)/(1+100\tilde{\epsilon})(1+\tilde{\epsilon})) - c_{13}} \\ & \geq c_{14} t^{N(x, \rho) (1-\epsilon) - c_6}.
		\end{align}

		This concludes the proof of the lower bound. The upper bound in Lemma \ref{lem:Doubling-Scaling} follows similarly. To show the additional statements in the Lemma, it suffices to take $ \epsilon/2 $ instead of $ \epsilon $ and require that
		\begin{equation}
			N(x, \rho) > \frac{2}{\epsilon} c_6 (\epsilon/2) =: N_0 (\epsilon/2).
		\end{equation}
	\end{proof}
	We will also need the following comparison for doubling numbers at nearby points (cf. Lemma 7.4, \cite{L}).
	\begin{lemma}
		There exists a radius $ r_0 > 0 $ and a threshold $ N_0 $ such that, for $ x_1, x_2 \in B(p, r) $ and a $ \rho > 0 $ such that $ d(x_1, x_2) < \rho < r_0, N(x_1, \rho) > N_0 $, there exists a constant $ c_{15} > 0$ such that
		\begin{equation}
			N(x_2, c_{15} \rho) > \frac{99}{100} N(x_1, \rho).
		\end{equation}
	\end{lemma}
	
	\begin{proof}
		The proof proceeds exactly as in Lemma 7.4, \cite{L}, using Lemma \ref{lem:Doubling-Scaling} above.
	\end{proof}
	
	\section{Additivity of frequency} \label{sec:Additivity-Frequency}
	
	Similarly to \cite{L}, we discuss the accumulation properties of the doubling index. The two main statements of this section are a barycenter estimate and a propagation of smallness result.
	
	
	\subsection{Barycenter accumulation}
	
	Roughly speaking, we will assert the following: suppose that the doubling exponents at the vertices $ \{ x_1, \dots, x_{n+1} \} $ of a simplex are large (i.e. bounded below by a fixed $ N_0 > 0 $). Then, the doubling exponent at the barycenter of the simplex $ x_0 := \frac{1}{n} \sum_{i = 1}^{n+1} x_i $ is bounded below by $ (1 + c) N_0$, where $ c > 0 $ is a fixed constant. Heuristically, the growth "accumulates" at the barycenter. We recall that here $n$ is the dimension of the Euclidean space in which we are working. The proof proceeds via direct use of the frequency properties discussed in Section \ref{sec:Tool-Box}.
	
	\begin{definition}
		Given a simplex $ S := \{ x_1, \dots, x_{n+1} \} $, we define the relative width $ w(S) $ of $ S $ as
		\begin{equation}
			w(S) := \frac{\width(S)}{\diam(S)},
		\end{equation}
		where $ \diam(S) $ is the diameter of $ S $ and $ \width(S) $ is the smallest possible distance between two parallel hyperplanes, containing $ S $ in the region between them.
	\end{definition}
	
	Further on, we will consider simplices $ S $ whose relative width is bounded below as $ w(S) \geq w_0 := w_0(n) > 0 $ - the specific bound $ w_0 $ will be specified later.
	
	Now, in order to apply the scaling of frequency we will need the following covering lemma.
	
	\begin{lemma} \label{lem:Covering-Lemma}
		There exist a constant $ \alpha := \alpha(n, w_0) > 0 $ and a radius $ \rho := \rho(n, w_0) $ with $ K := \frac{\rho}{\diam(S)} \geq \frac{2}{w_0} $ such that
		\begin{equation}
			B(x_0, (1+\alpha) \rho) \subset \cup_{i=1}^{n+1} B(x_i, \rho).
		\end{equation}
		Moreover, for $ t > 2 $ there exists $ \delta(t) \in (0, 1) $ with $ \delta(t) \rightarrow 0 $ as $ t \rightarrow \infty $, so that
		\begin{equation}
			B(x_i, t \rho) \subset B(x_0, (1+\delta)t \rho).
		\end{equation}
	\end{lemma}
	
	The main result of this subsection is the following proposition.
	\begin{proposition}
		Let $ \{B_i\}_{i=1}^{n+1} $ be a collection of balls centered at the vertices $ \{x_i\}_{i=1}^{n+1} $ of the simplex $ S $ and radii not exceeding $ \frac{\rho}{2} $, where $ \rho = \rho(n, w_0) $ comes from Lemma \ref{lem:Covering-Lemma}. Then, there exist positive constants $ c := c(n, w_0), C := C(n, w_0) \geq K, r := r(w_0, L)$ and $ N_0 := N_0(w_0, L) $ with the following property:
		
		If $ S \subset B(p, r) $ and if $ N(B_i) > N > N_0, i = 1, \dots n+1 $, then
		\begin{equation}
			N(x_0, C \diam S) > (1+c) N.
		\end{equation}
	\end{proposition}
	
	\begin{proof}
		First, Lemma \ref{lem:Doubling-Scaling} shows that by taking larger balls, the doubling exponents essentially increase, so we can assume that all balls $ B_i $ have the radius $ \rho $.
		
		Let us set
		\begin{equation}
			M := \sup_{\cup_{i=1}^{n+1} B(x_i, \rho)} |u|,
		\end{equation}
		
		and let us suppose that $ M $ is achieved on the ball $ B(x_{i_0}, \rho) $ for a fixed index $ i_0 $.
		
		In particular, by Lemma \ref{lem:Covering-Lemma} we have
		\begin{equation}
			\sup_{B(x_0, (1+\alpha) \rho)} |u| \leq M.
		\end{equation}
		
		Further, let us introduce parameters $ t > 2, \epsilon > 0 $ to be specified below and assume that the second statement in Lemma \ref{lem:Doubling-Scaling} holds for the ball $ B(x_{i_0}, t \rho) $, by which we see
		\begin{equation}
			\sup_{B(x_{i_0}, t \rho)} |u| \geq M t^{N(1-\epsilon)}.
		\end{equation}
		
		Moreover, assuming that the scaling in Lemma \ref{lem:Doubling-Scaling} holds at the barycenter $ x_0 $ and recalling Lemma \ref{lem:Covering-Lemma}, we conclude
		\begin{align}
			\left( \frac{t(1+\delta)}{1+\alpha} \right)^{N(x_0, t(1+\delta) \rho) (1+\epsilon) + c_6} &\geq \frac{\sup_{B(x_0, t(1+\delta) \rho)} |u|}{\sup_{B(x_0, (1+\alpha) \rho)} |u|} \geq \frac{\sup_{B(x_{i_0}, t \rho)} |u|}{\sup_{B(x_0, (1+\alpha) \rho)} |u|} \\ &\geq \frac{M t^{N(1-\epsilon)}}{M} = t^{N(1-\epsilon)}.
		\end{align}
		
		Specifying the parameters, we select $ t > 2 $ large enough to ensure $ \delta(t) \leq \frac{\alpha}{2} $, and hence
		\begin{equation}
			\frac{t(1+\delta)}{1+\alpha} \leq t^{1 - \gamma},
		\end{equation}
		for some $ \gamma = \gamma(t, \alpha) \in (0, 1) $. Thus, putting the last estimates together we see
		\begin{equation}
			t^{(1-\gamma) N(x_0, t(1+\delta) \rho) (1+\epsilon) + c_6} \geq t^{N(1-\epsilon)}
		\end{equation}
		and therefore
		\begin{equation}
			N(x_0, t(1+\delta) \rho) \geq \frac{1-\epsilon}{(1+\epsilon)(1-\gamma)} N - c_6.
		\end{equation}
		Selecting an $ \epsilon = \epsilon(\gamma) > 0 $ we can arrange that
		\begin{equation}
			\frac{1-\epsilon}{(1+\epsilon)(1-\gamma)} > 1 + 2 c,
		\end{equation}
		for some $ c := c(\gamma) > 0 $. Hence, we conclude
		\begin{equation}
			N(x_0, t(1+\delta) \rho) \geq N(1+2c) - c_6 \geq (1+c)N + (c N_0 - c_6) > (1+c)N,
		\end{equation}
		provided that $ N_0 $ is sufficiently big.
	\end{proof}
	\subsection{Propagation of smallness}
	
		We use propagation of smallness to derive estimate on the doubling exponents. The main auxiliary result in this discussion is the propagation of smallness for Cauchy data.
		
		\begin{lemma}[cf. Lemma 4.3, \cite{Li}] \label{lem:Propagation-Smallness}
			Let $ u $ be a solution of $ (\ref{eq_main_elliptic}) $ in the half-ball $ B^+_1 $ where the conditions $ (\ref{eq:Uniformly-Elliptic}), (\ref{eq:Uniformly-Bounded}), (\ref{eq:Lipschitz-Coef})$ are satisfied. Let us set
			\begin{equation}
				F := \{ (x', 0) \in \mathbb{R}^n| x' \in \mathbb{R}^{n-1}, |x'| < \frac{3}{4} \}.
			\end{equation}
			
			If the Cauchy conditions
			\begin{equation}
				\|u\|_{H^1(F)} + \| \partial_n u \|_{L^2(F)} \leq \epsilon < 1 \quad  \text{and} \quad \|u\|_{L^2( B^+_1)} \leq 1.
			\end{equation}
			are satisfied, then
			\begin{equation}
				\| u \|_{L^2(\frac{1}{2} B^+_1)} \leq c \epsilon^\beta,
			\end{equation}
			where the constants $ c, \beta $ depend on $ n, \eta, \Lambda, \Gamma $.
		\end{lemma}
		
		It is convenient to introduce the following doubling index.
		\begin{definition}
			The (uniform) doubling index $ N(Q) $ of a cube $ Q $ is defined as
			\begin{equation}
				N(Q) := \sup_{x \in Q, r \in (0, \diam(Q))} N(x, r).
			\end{equation}
		\end{definition}
		
		This doubling index enjoys the following monotonicity property
		\begin{equation}
			N(q) \leq N(Q), \quad \text{if} \quad q \subseteq Q.
		\end{equation}
		Also, if $ Q \subseteq \cup_i Q_i $ with $ \diam(Q) \leq \diam(Q_i) $, then there exists an index $ i_0 $ such that
		\begin{equation}
			N(Q) \leq N(Q_{i_0}).
		\end{equation}
		
		The next proposition roughly asserts that if a bunch of sub-cubes around a hyperplane all have a high doubling index, then a larger cube containing them must also have a big doubling index.
		
		\begin{proposition} (cf. Lemma 4.1, \cite{L})
			Let $ Q $ be a cube $ [-R, R] $ in $ \mathbb{R}^n $ and let us divide $ Q $ into $ (2A + 1)^n $ equal sub-cubes $ q_i $ with side-length $ \frac{2R}{2A+1} $. Let $ \{q_{i, 0}\} $ be the collection of sub-cubes which intersect the hyperplane $ \{ x_n = 0 \} $ and suppose that there exist centers $ x_i \in q_{i, 0} $ and radii $ r_i < 10 \diam(q_{i,0}) $ so that $ N(x_i, r_i) > N $ where $ N $ is fixed. Then there exist constants $ A_0 = A_0(n), R_0 = R_0(L), N_0 = N_0(L) $ with the following property:
			
			If $ A > A_0, N > N_0, R < R_0 $, then
			\begin{equation}
				N(Q) > 2N.
			\end{equation}
		\end{proposition}
		
		\begin{proof}
			We assume that $ R_0 $ is small enough, so that Lemma \ref{lem:Doubling-Scaling} holds with $ \epsilon = 
			\frac{1}{2} $ and the equation $ (\ref{eq_main_elliptic}) $ at this scale is satisfied along with the conditions $ (\ref{eq:Uniformly-Elliptic}), (\ref{eq:Uniformly-Bounded}), (\ref{eq:Lipschitz-Coef})$. Moreover, at this scale we can also use Lemma \ref{lem:Propagation-Smallness}.\\
			
			To ease notation, without loss of generality by scaling we may assume that $ R = \frac{1}{2}, R_0 \geq \frac{1}{2} $. Let $ B $ be the unit ball centered at $ 0 $. We consider the half ball $ \frac{1}{32} B^+ \subset \frac{1}{8} B$ and wish to apply the propagation of smallness for Cauchy data problems. To this end, we need to bound $ u $ and $ \nabla u $ on $ F := \frac{1}{32} B^+ \cap \{ x_n = 0 \} $.

			\textbf{Step 1 - Bound on $ u $}.
			
			First, let us set
			\begin{equation}
				M := \sup_{\frac{1}{8}B} |u|,
			\end{equation}
			by which we have
			\begin{equation}
				\sup_{B(x_i, \frac{1}{32})} |u| \leq M, \quad \forall x_i \in \frac{1}{16} B.
			\end{equation}
			
			Hence, for $ x_i \in \frac{1}{16} B $, Lemma \ref{lem:Doubling-Scaling} and the assumption that $ N(x_i, r_i) > N $ imply
			\begin{equation}
				\sup_{8 q_{i, 0}} |u| \leq \sup_{B(x_i, \frac{16 \sqrt{n}}{2A+1})} |u| \leq c_{16} \left( \frac{512 \sqrt{n}}{2A+1} \right)^{\frac{N}{2}} \sup_{B(x_i, \frac{1}{32})} |u| \leq e^{-c_{17} N \log A} M,
			\end{equation}
			where $ c_{17} = c_{17}(n) > 0 $ and we have assumed in the last step that $ N, A $ are sufficiently large.

			\textbf{Step 2 - Bound on $ \nabla u $}.
			
			Further, we wish to bound the gradient $ |\nabla u| $. We recall the following facts.
			
			\begin{lemma} \label{lem:Sobolev-Regularity}
				Let $ u $ be a solution of equation $ (\ref{eq_main_elliptic}) $ in a domain $ \Omega $ satisfying the conditions $ (\ref{eq:Uniformly-Elliptic}), (\ref{eq:Uniformly-Bounded}), (\ref{eq:Lipschitz-Coef})$. Then, if $ \Omega' \subset \subset \Omega $, we have
				\begin{equation}
					\|u\|_{W^{2,2}(\Omega')} \leq c_{18} \|u\|_{L^2(\Omega)},
				\end{equation}
				where $ c_{18} > 0 $ depends on the parameters in (\ref{eq:Uniformly-Elliptic}), (\ref{eq:Uniformly-Bounded}), (\ref{eq:Lipschitz-Coef}) and $ d(\Omega', \Omega) $.
			\end{lemma}
			
			For a proof of Lemma \ref{lem:Sobolev-Regularity} we refer to Theorem $ 8.8 $, the remark thereafter and Problem $ 8.2 $, \cite{GT}. We also observe that if $ \Omega', \Omega $ are replaced by small concentric cubes $ Q_r, Q_{2r} $, then by scaling a factor of $ \frac{1}{r^2} $ appears on the right hand side.
			
			\begin{lemma} \label{lem:Gradient-Restriction}
				Let $ u \in W^{2,2}(\mathbb{R}^n) $ and let us consider the trace of $ u $ onto the hyperplane $ \{ x_n = 0 \} \cong \mathbb{R}^{n-1} $ which, abusing of notation, we also denote by $ u $. Then
				\begin{equation}
					\| \nabla u \|_{L^2(\mathbb{R}^{n-1})} \leq c_{19}(\|u\|_{W^{2,2}(\mathbb{R}^n)} + \|u\|_{L^2(\mathbb{R}^{n-1})} ),
				\end{equation}
				where $ c_{19} = c_{19}(n) $.
			\end{lemma}
			
			For a proof of Lemma \ref{lem:Gradient-Restriction} we refer to Lemma $ 23 $, \cite{BL}. Using Lemma \ref{lem:Gradient-Restriction} for functions of the form $ \chi u $, where $ \chi $ is a standard smooth cut-off function and $ u \in W^{2,2} $ we see that			
			\begin{equation}
				\| \nabla u \|_{L^2(\mathbb{R}^{n-1} \cap B_r )} \leq c_{19}(\|u\|_{W^{2,2}(B_{2r})} + \|u\|_{L^2(\mathbb{R}^{n-1} \cap B_{2r})} ),
			\end{equation}
			
			where $ \chi $ is supported in $ B_{2r} $.
			
			We now recall the following standard Sobolev trace estimate (see \cite{E}, Section 5.5, Theorem 1). If $U$ is bounded and $\partial U $ is $C^1$, then there holds
			\begin{equation}
			||u||_{W^{\frac{3}{2}, 2}(\partial U)} \leq c_{20} ||u||_{W^{2,2}(U)},
			\end{equation}
			where the positive constant $c_{20}$ depends only on the domain $U$. We notice that $$\text{dist}(4q_{i,0}, \partial (8 q_{i,0})) = \frac{4}{2A+1},$$ so that,
			using the last lemmas along with the trace estimate, we have
			\begin{align}
				\|\nabla u\|_{L^2(F \cap q_{i,0})} & \leq c_{19} \left( \frac{2A+1}{4}\right) (\|u\|_{W^{2,2}(2 q_{i,0})} + \|u\|_{L^2(F \cap 2 q_{i,0})} ) \\ &\leq (2c_{19}c_{20}) \left( \frac{2A+1}{4}\right)^3 \|u\|_{W^{2,2}(4q_{i,0})} \leq c_{21} \left( \frac{2A+1}{4}\right)^5 \|u\|_{L^2(8q_{i,0})}.
			\end{align}
			Here, $c_{21} = 2 (c_{18}c_{19}c_{20})$.
			Again using the trace estimate, this shows that
			\begin{align}
				\|u\|_{W^{1,2}(F \cap q_{i,0})} + \| \partial_n u \|_{L^2(F \cap q_{i,0})} &\leq c_{20} \left( \frac{2A+1}{4}\right)^2 \|u\|_{W^{2,2}(4q_{i,0})} + \|\nabla u\|_{L^2(F \cap q_{i,0})}  \\
				&\leq (c_{20}c_{21}) \left( \frac{2A+1}{4}\right)^5 \|u\|_{L^2(8q_{i,0})} \\ 
				&\leq \frac{c_{22}}{(2A+1)^{n}} \left( \frac{2A+1}{4}\right)^5 \sup_{8q_{i,0}} |u|.
			\end{align}
			
			Summing up over the cubes $ q_{i,0} $ and using the bound in the first step, we get
			\begin{equation}
				\|u\|_{W^{1,2}(F)} + \| \partial_n u \|_{L^2(F)} \leq \frac{c_{23}}{(2A+1)^{n-2}} \left( \frac{2A+1}{4}\right)^5 \sup_{8q_{i,0}} |u| \leq e^{-c_{17} N \log A} M.
			\end{equation}
			
			\textbf{Step 3 - Propagation of smallness}.
			
			Let us observe that
			\begin{equation}
				\|u\|_{L^2(\frac{1}{32} B^+)} \leq c_{24} M.
			\end{equation}
			and set
			\begin{equation}
				v := \frac{u}{c_{24} M},
			\end{equation}
			by which we have
			\begin{equation}
				\|v\|_{L^2(\frac{1}{32} B^+)} \leq 1.
			\end{equation}
			
			Hence, by the bounds in Steps $ 1 $ and $ 2 $ and propagation of smallness from Lemma \ref{lem:Propagation-Smallness} we get
			\begin{equation}
				\|v\|_{L^2(\frac{1}{64} B^+)} \leq \epsilon^\beta,
			\end{equation}
			where $ \epsilon = e^{-c_{17} N \log A} $.
			
			Let us select a ball $ B(p,\frac{1}{256}) \subset \frac{1}{64} B^+ $ and observe that by (\ref{eq:Elliptic-Regularity})
			\begin{equation}
				\sup_{B(p,\frac{1}{256})} |v| \leq \epsilon^\beta,
			\end{equation}
			which implies
			\begin{equation}
				\sup_{B(p,\frac{1}{256})} |u| \leq  e^{-c_{25} \beta N \log A} M.
			\end{equation}
			
			Moreover, as $ \frac{1}{8} B \subset B(p, \frac{1}{2}) $, we have by definition $ \sup_{B(p, \frac{1}{2})} |u| \geq M $. This implies
			\begin{equation}
				\frac{\sup_{B(p, \frac{1}{2})} |u|}{\sup_{B(p, \frac{1}{256})} |u|} \geq e^{c_{25} \beta N \log A}.
			\end{equation}
			Finally, applying the doubling scaling Lemma \ref{lem:Doubling-Scaling} we have
			\begin{equation}
				\frac{\sup_{B(p, \frac{1}{2})} |u|}{\sup_{B(p, \frac{1}{256})} |u|} \leq (128)^{\tilde{N}/2},
			\end{equation}
			where $ \tilde{N} $ is the doubling index for $ B(p, \frac{1}{2}) $. Therefore,
			\begin{equation}
				\tilde{N} \geq c_{26} N \log A \geq 2N,
			\end{equation}
			where $ A $ is assumed to be sufficiently large.
		\end{proof}
	\section{Counting Good/Bad cubes and application to nodal geometry}
	
	Using the results of Section \ref{sec:Additivity-Frequency}, one can deduce the following result.
	
	\begin{theorem}
		There exist constants $ c > 0 $, an integer $ A $ depending on the dimension $ d $ only and positive numbers $ N_0 = N_0(M, g), r = r(M, g) $ such that for any cube $ Q \in B(p, r) $ the following holds:
		
		If $ Q $ is partitioned into $ A^n $ equal sub-cubes $ q_i $, then
		\begin{equation}
			\# \{q_i| N(q_i) \geq \max( \frac{N(Q)}{1+c}, N_0)  \} \leq \frac{A^{n-1}}{2}.
		\end{equation}
	\end{theorem}
	
	The proof is combinatorial in nature and we refer to Theorem $ 5.1 $, \cite{L} for complete details. As an application of the previous theorem, we also have our main theorem
	
	\begin{theorem}\label{thm:Main}
		There exist positive numbers $ r_0 = r_0(M, g), c = c(M, g) $ and $ \alpha = \alpha(n) $ such that for any solution $ u $ of equation $ (\ref{eq_main_elliptic}) $ in a domain $ \Omega $ satisfying the conditions $ (\ref{eq:Uniformly-Elliptic}), (\ref{eq:Uniformly-Bounded}), (\ref{eq:Lipschitz-Coef})$, we have
		\begin{equation}
			\mathcal{H}^{n-1}(\{u = 0\} \cap Q) \leq c \diam^{n-1}(Q) N^\alpha (Q),
		\end{equation}
		where $ Q \subset B(p, r_0) $ is an arbitrary cube in $\Omega$.
	\end{theorem}
	
	For details, we refer to Theorem 6.1, \cite{L}.
	
	\section{Application to Steklov eigenfunctions}\label{sec:Steklov}
	
	Our goal is to transform a solution $\phi_\lambda$ to the Steklov problem (\ref{Eq:main-Steklov}) on a manifold $M$ into a solution $u$ to Equation (\ref{eq_main_elliptic}) on some domain $\Omega \subset \mathbb{R}^n$. 
	
	\subsection{Getting rid of the boundary}
	
	There exists a procedure (see \cite{BL, Zhu1, Zhu2}) to transform $M$ into a compact manifold without boundary, which we highlight here. We first let $d(x) := \text{dist}(x, \partial M)$ be the distance between a point $x \in M$ and the boundary. We then define 

\begin{equation}
\delta(x)=
\begin{cases}
d(x) & x  \in M_\rho,\\
l(x) & x \in M \setminus M_\rho,
\end{cases}
\end{equation}
where $\rho = \rho(M) > 0$ is such that $d(x)$ is smooth in a $\rho$ neighborhood $M_\rho$ of $\partial M$ in $M$. We choose $l \in C^\infty(M \setminus M_\rho)$ in such a way that makes $\delta$ smooth on $M$. It now follows that 
\begin{equation}\label{eq:ext-steklov-egnfunction}
v(x) := e^{\lambda \delta(x)}  \phi_\lambda(x),
\end{equation}
 identifies with $\phi_\lambda$ on $M$  and satisfies a Neumann boundary condition. More precisely, $v$ solves
\begin{equation}
\begin{cases}
\Delta_g v + b(x) \cdot \nabla_g v + q(x) v = 0 & \text{in } M, \\
\partial_\nu v = 0 & \text{on } \partial M,
\end{cases}
\end{equation}
where $\displaystyle \nu = - \nabla \delta $ is the unit outward normal and with
\begin{equation}
\begin{cases}
b(x) = -2 \lambda \nabla_g \delta(x), \\
q(x) = \lambda^2 | \nabla \delta(x)|^2 - \lambda \Delta_g \delta(x).
\end{cases}
\end{equation}
The fact that $v$ satisfies a Neumann boundary condition now allows us to get rid of the boundary by gluing to copies of $M$ together along the boundary and extend $v$ in the natural way. Denote by $\bar{M} = M \cup M$ the compact boundaryless manifold obtained by doing so. We remark that the induced metric $\bar{g}_{ij}$ on $\bar{M}$ is Lipschitz on $\partial M$. Using the canonical isometric involution that interchanges the two copies $M$ of $\bar{M}$, we can then extend $v, b$ and $q$ to $\bar{M}$. Abusing notation and writing $v$ for the extension, we obtain that $v$ satisfies the elliptic equation
\begin{equation}\label{eq_main_eq}
\Delta_{\bar{g}} v + \bar{b} (x) \cdot \nabla_{\bar{g}} v + \bar{q}(x) v = 0
\end{equation}
in $\bar{M}$ and we have the following bounds 
\begin{equation}\label{eq:steklov-coeff-bounds}
\begin{cases}
 || \bar{b} ||_{W^{1, \infty}(N)} \leq C \lambda, \\
 || \bar{q} ||_{W^{1, \infty}(N)} \leq C \lambda^2. \\
\end{cases}
\end{equation}

Fix a point $O$ in $\bar{M}$. In local coordinates around $O$, we have 
\begin{equation}\label{eq:laplace-beltrami-local}
\Delta_{\bar{g}}f = \frac{1}{\sqrt{|\bar{g}|}} \partial_i (\sqrt{|\bar{g}|}\bar{g}^{ij} \partial_j f), \quad (\nabla_{\bar{g}} f )^i = \bar{g}^{ij}\partial_j f.
\end{equation}
where $\sqrt{|\bar{g}|}$ is the determinant of the extended metric tensor $\bar{g}$. Since the extended metric is Lipschitz and recalling the boundedness of $\bar{b}$ and $\bar{q}$, it then  follows that $v$ is a solution of Equation (\ref{eq_main_elliptic}) with $L$ satisfying the conditions (\ref{eq:Uniformly-Elliptic}, \ref{eq:Uniformly-Bounded},  \ref{eq:Lipschitz-Coef}). \\

In order to get uniform control over the coefficients, we now work at wavelength scale and  consider the ball $B(x_0, 1/ \lambda) \subset \bar{M}$. We introduce $$v_{x_0, \lambda}(x) := v(x_0 + \frac{x}{\lambda}),$$ for $x \in B(0, 1)$. Then, $v_{x_0, \lambda}$ satisfies Equation (\ref{eq_main_elliptic}) where the coefficients $(a^{ij}), b^i$ and $c$ are uniformly bounded in $L^\infty$ by a constant not depending on $\lambda$. Moreover, the ellipticity constant of the $(a^{ij})$ does not change and the Lipschitz constant $\Gamma$ can only improve. In clear, the family of $v_{x_0, \lambda}$ solves Equation \ref{eq_main_elliptic} and satisfies the conditions $ (\ref{eq:Uniformly-Elliptic}), (\ref{eq:Uniformly-Bounded}), (\ref{eq:Lipschitz-Coef})$ without any dependence on $\lambda$. In what follows, we will thus be able to apply Theorem \ref{thm:Main} uniformly on this family. For more details on the above, we refer the reader to Section 3.2 of \cite{BL}.

\subsection{Upper bound for the nodal set}

\begin{remark}
Many of the results needed we collect in this subsection work only within a small enough scale $r < r_0$. Since we work locally at wavelength scale $r = \frac{1}{\lambda}$, all those results hold for $\lambda$ big enough.
\end{remark}

We now fix a point $x_0$ in $\bar{M}$, let $r_0 = \lambda^{-1}$ and choose normal coordinates in a geodesic ball $B_{\bar{g}} (x_0, r_0)$. Without loss of generality, we assume $r_0$ is smaller than the injectivity radius of $\bar{M}$. For $x,y$ in $B_{\bar{g}}(x_0,r_0)$, we respectively denote the Euclidean and Riemannian distance by $d(x,y)$ and $d_{\bar{g}}(x,y)$. For $\lambda$ big enough, we have
\begin{equation}\label{eq:metric-comp}
d_{\bar{g}}(x,y) \leq 2 d(x,y)
\end{equation}
for any two distinct points $x,y \in B_{\bar{g}}(x_0, r_0)$. By construction, the nodal sets of the eigenfunction $\phi_\lambda$ and its extension $v$ coincide in $M$. Combining this observation with Equation (\ref{eq:metric-comp}) allows to compare the size of the corresponding nodal sets on small balls. Indeed,  for any $r < r_0/2$, one has
\begin{equation}\label{eq:nodal-comp}
\mathcal{H}^{n-1}(Z_{\phi_\lambda} \cap B_{\bar{g}}(O, r)) \leq \mathcal{H}^{n-1}(Z_v \cap B(x, 2r )) 
\end{equation}
Denoting by $Z_{v_{x_0, \lambda}}$ the nodal set of $v_{x_0, \lambda}$, we then remark that
\begin{equation}
\mathcal{H}^{n-1}(Z_v \cap B(x, 2r )) \leq \lambda^{1-n} \mathcal{H}^{n-1}(Z_{v_{x_0, \lambda}})
\end{equation}

Also, by Proposition 1 in \cite{Zhu2}, there exists $c_1 > 0$ such that the doubling index of $N_{x_0, \lambda}(x, r)$ of $v_{x_0, \lambda}$ on the ball $B(x,r) \subset B(0,1)$ satisfies 
\begin{equation}\label{eq:zhu-growth-bound}
N_{x_0, \lambda}(x,r) \leq c_1 \lambda
\end{equation}
for any $r < r_0$. We choose $r < r_0/4$ and let $Q$ be the cube centered at origin and of side length $r$ so that the above now implies 
\begin{equation}
N_{x_0, \lambda}(Q) = \sup_{x \in Q, r \in (0, \diam(Q))} N_{x_0, \lambda}((x,r) \leq c_1 \lambda.
\end{equation}
Collecting all of the above, noticing that $B(0,2r) \subset Q$ and using Theorem \ref{thm:Main}, we finally get that 
\begin{align*}
\mathcal{H}^{n-1}(Z(\phi_\lambda) \cap B_{\bar{g}}(x_0, r)) &\leq \lambda^{1-n} \mathcal{H}^{n-1}(Z_{v_{x_0, \lambda}} \cap Q) \\
&\leq c_1(n) \lambda^{1-n}  N^\alpha(Q) \\
&\leq c_2(n) \lambda^{\alpha - n + 1}.
\end{align*}

Covering $M$ with $\sim \lambda^{n}$ balls $B(x_0,r)$ of radius $r = \frac{1}{4\lambda}$ finally yields 
\begin{equation}
\mathcal{H}^{n-1}(Z_\lambda) \leq c \lambda^{\alpha+1}
\end{equation}
and thus concludes the proof of Theorem \ref{thm:Steklov-Bound}.


\end{document}